\documentclass[12pt,leqno]{amsart}
\usepackage{epsfig}
\usepackage{amssymb}
\usepackage{amscd}
\usepackage[all]{xy}

\newtheorem{theo}{Theorem}[section]
\newtheorem{lemma}[theo]{Lemma}
\newtheorem{defi}[theo]{Definition}
\newtheorem{prop}[theo]{Proposition}

\newtheorem{cor}[theo]{Corollary}
\newtheorem{remark}[theo]{Remark}

\newtheorem{example}[theo]{Example}
\numberwithin{equation}{section}

\def\pre-tr{\operatorname{pre-tr}}

\def\Hom{\operatorname{Hom}}
\def\End{\operatorname{End}}

\newcommand{\bbG}{{\mathbb G}}

\newcommand{\cO}{{\mathcal O}}

\newcommand{\cL}{{\mathcal L}}
\newcommand{\cM}{{\mathcal M}}

\newcommand{\cA}{{\mathcal A}}
\newcommand{\cB}{{\mathcal B}}
\newcommand{\cI}{{\mathcal I}}

\newcommand{\cR}{{\mathcal R}}

\newcommand{\cH}{{\mathcal H}}

\newcommand{\Sets}{\operatorname{Sets}}

\newcommand{\id}{\operatorname{id}}

\newcommand{\Sch}{\operatorname{Sch}}
\newcommand{\Rep}{\operatorname{Rep}}

\title[Hilbert schemes of points for associative algebras]
{Hilbert schemes of points for associative algebras}

\author{Michael~Larsen}
\address{Department of Mathematics, Indiana University,
Bloomington, IN 47405, USA} \email{mjlarsen@indiana.edu}
\author{Valery A.~Lunts}
\address{Department of Mathematics, Indiana University,
Bloomington, IN 47405, USA} \email{vlunts@indiana.edu}

\thanks{The first named author was partially supported by the NFS grant 1101424.
The second named author was partially supported by the NSF grant
0901301}

\thanks{}

\begin{document}

\begin{abstract} For a finitely generated associative algebra $\cA$ over a
commutative ring $k$ we construct the Hilbert scheme ${\bf
H}^{[n]}_{\cA}$ which parametrizes left ideals in $\cA$ of
codimension $n.$
\end{abstract}

\maketitle

\section{Introduction}

Fix a commutative ring  $k.$ We denote by $F_m=k \langle
x_1,...,x_m\rangle$ the free associative $k$-algebra on $m$
generators. For a two-sided ideal $\cR \subset F_m$ consider the
quotient $k$-algebra $\cA =F_m/\cR.$ We are going to construct the
"Hilbert scheme ${\bf H}^{[n]}_{\cA}$ of $n$-points" in
"$\mathrm{Spec} \cA,$" i.e. the moduli space of left ideals in
$\cA$ of codimension $n.$ As in the usual case of a commutative
algebra $\cA$ this Hilbert scheme represents the corresponding
functor. We proceed in two steps: first we construct the scheme
$\tilde{{\bf H}}^{[n]}_{\cA},$ which we call the "based Hilbert
scheme". It parametrizes left ideals $I\subset \cA$ with a choice
of a basis in the cyclic module $\cA /I.$ This scheme $\tilde{{\bf
H}}^{[n]}_{\cA}$ is quasi-affine and carries a natural free
$\mathrm{GL}_n$-action, so that the quotient scheme is ${\bf
H}^{[n]}_{\cA}.$ Thus the natural map
$$\pi :\tilde{{\bf H}}^{[n]}_{\cA}\to {\bf H}^{[n]}_{\cA}$$
is a Zariski locally trivial $\mathrm{GL}_n$-bundle. We do not use
Geometric Invariant Theory to find ${\bf H}^{[n]}_{\cA};$ all our
constructions are explicit.

We then proceed to show that ${\bf H}^{[n]}_{\cA}$ is naturally a
projective scheme (over an appropriate affine scheme).

In the last section we have collected some facts about free
actions of group
schemes and equivariant quasi-coherent sheaves.

Hilbert schemes as ours we considered first by Van den Bergh
\cite{VdB} for the free algebra $F_m$ over a field. He shows in
particular that ${\bf H}^{[n]}_{F_m}$ is smooth and constructs for
it a natural cell decomposition. Later similar (or close)
constructions appeared also in
\cite{LeB},\cite{LeB-Se},\cite{Rei}, \cite{En-Rei}. The case of
graded (very general) Hilbert schemes was teated in \cite{Ar-Zh}.

We thank Michel Van den Bergh and Markus Reineke for pointing out
to us the above references. We also thank Tony Pantev for his
suggestion that our work may be related to \cite{Ar-Zh} (at the
moment we do not know if there is a relation).

\section{Construction of the Hilbert scheme for associative algebras}
 Denote by $\Sch
_k$ the category of $k$-schemes.

Fix a finitely generated associative $k$-algebra $\cA.$ For $X\in
\Sch _k$ we denote by $\cA_{X}$ the sheaf of $\cO _X$-algebras which
is associated to the presheaf $\cO _X\otimes _{k}\cA.$

\subsection{Based Hilbert scheme $\tilde{\bf{H}}^{[n]}_{\cA}$}

\begin{defi} Denote by $\tilde{{\bf M}}^{[n]}_{\cA}:\Sch _k \to \Sets$
the contravariant functor from the category of $k$-schemes to the
category of sets, which is defined as follows. For a scheme $X$ the
set $\tilde{{\bf M}}^{[n]}_{\cA}(X)$ is the set of equivalence
classes of triples $(M,v,B)$ where $M$ is a left $\cA_{X}$-module
which is free of rank $n$ as $\cO _X$-module, $B\subset \Gamma
(X,M)$ is a basis of the $\cO _X$-module $M,$
and  $v\in \Gamma (X,M)$ generates $M$ as
an $\cA_{X}$-module. The equivalence relation is the obvious one.

A morphism of $k$-schemes $f:Y\to X$ induces an isomorphism of
sheaves of algebras
$$f^*\cA_{X}=\cO _Y\otimes _{f^{-1}(\cO _X)}f^{-1}\cA_{X}\to
\cA_{Y}$$ which gives the map of sets $f^*:\tilde{{\bf
M}}^{[n]}_{\cA}(X)\to \tilde{{\bf M}}^{[n]}_{\cA}(Y)$.
\end{defi}

\begin{theo} \label{representability-based-hilbert} Fix $n\geq 1.$
The functor $\tilde{{\bf M}}^{[n]}_{\cA}$ is represented by a
quasi-affine $k$-scheme $\tilde{{\bf H}}^{[n]}_{\cA}.$ If the ground
ring $k$ is noetherian, then $\tilde{{\bf H}}^{[n]}_{\cA}$ of finite
type over $k.$
\end{theo}

\begin{proof} a) Consider the polynomial ring $k[t_{ij}^s]$ in the
set of $mn^2$ variables $t_{ij}^s,$ where $1\leq s\leq m,1\leq
i,j\leq n.$ Put $W=\mathrm{Spec} k[t_{ij}^s].$  For each $s$
denote by $A_s=(t^s_{ij})$ the square $n\times n$ matrix with
$t^s_{ij}$ as the $(i,j)$-entry. Put $V=\mathrm{Spec}
k[y_1,...,y_n]$ and let $Y$  denote the column vector
$Y=(y_1,...,y_n)^t.$ It will be convenient for us to arrange the
linear coordinates on the scheme $W\times V$ as $m$-tuples of
matrices and a column vector $(A_1,...,A_m,Y).$
 For each ordered $(n-1)$-tuple of elements $f_1,...,f_{n-1}\in
F_m$ define the open subscheme $U_{f_1...f_{n-1}} \subset
W\times V$ by the equation
\begin{equation} \label{det-neq-zero}
D _{f_1...f_{n-1}}:=\det
(Y,f_1(A_1,...,A_m)Y,...,f_{n-1}(A_1,...,A_m)Y)\neq 0
\end{equation}
I.e. the coordinate ring of the affine scheme $U_{f_1...f_{n-1}}$ is
the localization of the coordinate ring of $W\times V$ at the
element $D _{f_1...f_{n-1}}.$ Let ${\bf U}\subset W\times V$ be the
open subscheme which is the union of $U_{f_1...f_{n-1}}$'s for
all $f_1,...,f_{n_1}\in F_m.$ Denote by $Z=W\times V-{\bf U}$ the
complementary closed subscheme which is defined by the simultaneous
vanishing of the determinants $D_{f_1...f_{n-1}}.$   If
$k$ is noetherian, so is the scheme $W\times V,$ hence in this case
${\bf U}$ is the union of a finite number of $U_{f_1...f_{n-1}}$'s.

 For every relation $r(x_1,...,x_m)\in \cR$ in the
algebra $\cA$ consider the corresponding matrix $r(A_1,...,A_m).$
Denote by $\Rep _{\cA}^{[n]} \subset W$ the closed subscheme defined by
setting $r(A_1,...,A_m)=0$ for all $r\in \cR.$ (So the $S$-points of
$\Rep _{\cA}^{[n]}$ correspond to $k$-algebra homomorphisms $\cA \to
M_n(S)$ for any commutative $k$-algebra $S.$)

Finally define $\tilde{{\bf H}}^{[n]}_{\cA}$ as the scheme ${\bf U}
\cap (\Rep _{\cA}^{[n]}\times V).$ Clearly $\tilde{{\bf H}}^{[n]}_{\cA}$
is quasi-affine. If $k$ is noetherian then $\tilde{{\bf
H}}^{[n]}_{\cA}$ is of finite type over $k.$ We claim that is
represents the functor $\tilde{{\bf M}}^{[n]}_{\cA}.$

First we construct the universal triple $(M_0,v_0,B_0)$ over
$\tilde{{\bf H}}^{[n]}_{\cA}.$ Namely consider the free (left)
$k[t^s_{ij},y_l]$-module $M_0=k^n\otimes _kk[t^s_{ij},y_l]$ of rank
$n$ with the basis $B_0=\{e_l\otimes 1\}$ given by the standard
basis $\{e_l\}$ of $k^n,$ and the distinguished element $v_0=\sum
e_l\otimes y_l.$ It will be convenient for us to think of an element
$\sum e_l\otimes \alpha _l\in M_0$ (with $\alpha _l\in
k[t^s_{ij},y_l]$) as the dot product (i.e. the matrix product)
$${\bf e}\bullet \alpha ^t=
(e_1,...,e_n)\bullet (\alpha _1,...,\alpha _n)^{t}$$ of the row
vector ${\bf e}=(e_1,...,e_n)$ and the column vector ${\bf
\alpha}^t=(\alpha _1,...,\alpha _n)^{t}.$ In particular $v_0= {\bf
e}\bullet Y.$

The free algebra $F_m$ acts by endomorphisms of
this module $k[t^s_{ij},y_l]$-module $M_0$
via the matrices $A_s.$ Namely in the above
notation
$$x_s({\bf e}\bullet {\bf \alpha }^t)={\bf e}\bullet A_s{\bf \alpha }^t.$$
(Here $A_s\alpha ^t$ is the matrix product of the square matrix
$A_s$ with the column vector $\alpha ^t.$)
 If we
restrict this module to the closed
 subscheme $\Rep _{\cA}^{[n]}\times V,$ then the $F_m$-action descends
to the $\cA$-action. Finally if we further restrict to the open
subset $\tilde{{\bf H}}^{[n]}_{\cA},$ then the element
$v_0$ will be a generator of this $\cA _{\tilde{{\bf
H}}^{[n]}_{\cA}}$-module. We denote again by
$(M_0,v_0,B_0)$ the resulting
triple over $\tilde{{\bf H}}^{[n]}_{\cA}.$

This universal triple over $\tilde{{\bf H}}^{[n]}_{\cA}$ defines an
obvious morphism of the functor represented by the scheme
$\tilde{{\bf H}}^{[n]}_{\cA}$ to the functor $\tilde{{\bf
M}}^{[n]}_{\cA}.$ To show that this morphism is an isomorphism we
will construct the inverse map. Namely let $X$ be a $k$-scheme with
a triple $(M,v,B).$ Each generator $x_s$ of the algebra $\cA$ acts
on the basis $B$ by an $n\times n$-matrix with values in $\Gamma
(X,\cO _X).$ Thus to each variable $t^s_{ij}$ we associate a global
function on $X,$ which is the $(i,j)$-entry of this matrix. Now
express the element $v$ as a $\Gamma (X,\cO _X)$-linear combination
$v=\sum \gamma _l b_l$ of the vectors $b_l$ in $B$ and associate to
$y_l$ the element $\gamma _l.$ One checks that this association
defines a morphism of schemes $f:X\to \tilde{{\bf H}}^{[n]}_{\cA}$
so that the pullback of the universal triple is isomorphic to
$(M,v,B).$ This defines the inverse morphism of functors and proves
the theorem.
\end{proof}

\begin{example} \label{example-1}
Let $n=1.$ Then $\Rep ^{[1]}_{\cA}=Spec \cA _{ab},$ where $\cA
_{ab}=\cA /(\cA [\cA ,\cA]\cA)$ is the abelianization of the algebra
$\cA;$ $W\times V=Spec k[t^1,...,t^m,y]$ and ${\bf U}=\{y\neq 0\}.$
So $\tilde{{\bf H}}^{[1]}_{\cA}=Spec \cA _{ab}\times \bbG _m.$
\end{example}

\begin{prop} \label{extension-scalars-based-hilbert}
In the above notation let $k\to k^\prime $ be a homomorphism of
commutative algebras. Consider the algebra $\cA ^\prime=\cA \otimes
_{k }k ^\prime.$ Then for any $n\geq 1$
$$\tilde{\bf{H}}^{[n]}_{\cA ^\prime}=\tilde{\bf{H}}^{[n]}_{\cA}
\times_kk^\prime.$$
\end{prop}

\begin{proof} For the purpose of this proof denote by $X^\prime$ the
$k^\prime$-scheme obtained from a $k$-scheme $X$ by extension of
scalars from $k$ to $k^\prime.$ We use the notation of the proof of
Theorem \ref{representability-based-hilbert}.

First it is clear that $\Rep ^{[n]}_{\cA ^\prime}=
(\Rep ^{[n]}_{\cA})^\prime.$  Hence
$(\Rep ^{[n]}_{\cA}\times _k V)^\prime=
\Rep ^{[n]}_{\cA ^\prime}\times _{k^\prime} V^\prime .$

Next recall that the $\tilde{\bf{H}}^{[n]}_{\cA}$ is the open
subscheme of $\Rep ^{[n]}_{\cA}\times _k V$ which is the complement
of the closed subscheme $Z$ defined by the ideal generated by
elements $D_{f_1...f_{n-1}}.$ Extending the scalars to $k^\prime$ we
see that the closed subscheme $Z^\prime \subset (\Rep
^{[n]}_{\cA}\times _k V)^\prime= \Rep ^{[n]}_{\cA ^\prime}\times
_{k^\prime} V^\prime$ is the complement of
$\tilde{\bf{H}}^{[n]}_{\cA ^\prime}$ (since the determinants are
multi-linear in columns).  Hence it follows that
$(\tilde{\bf{H}}^{[n]}_{\cA})^\prime =\tilde{\bf{H}}^{[n]}_{\cA
^\prime }.$
\end{proof}

\begin{remark} \label{surjection=closed-embedding}
Let $\phi :\cA \to \cB$ be a surjective homomorphism
of (finitely generated associative) $k$-algebras.
This induces an obvious
morphism of functors $\phi _*:\tilde{\bf{M}}^{[n]}_{\cB}\to
\tilde{\bf{M}}^{[n]}_{\cA}$ by restriction of scalars from $\cB$ to
$\cA.$ This morphism of functors corresponds to a morphism of
representing schemes $\phi _*:\tilde{\bf {H}}^{[n]}_{\cB}\to
\tilde{\bf {H}}^{[n]}_{\cA}.$ It follows easily from the description
the scheme $\tilde{\bf{H}}^{[n]}$ in Theorem
\ref{representability-based-hilbert}  that $\phi _*$ is a closed
embedding.
\end{remark}

\subsection{Hilbert scheme $\bf{H}^{[n]}_{\cA}$}

\begin{defi} \label{defi-of-hilbert-functor}
Denote by ${\bf M}^{[n]}_{\cA}:\Sch _k \to \Sets$ the contravariant
functor from the category of $k$-schemes to the category of sets,
which is defined as follows. For a scheme $X$ the set ${\bf
M}^{[n]}_{\cA}(X)$ is the set of equivalence classes of pairs
$(M,v),$  where $M$ is a left $\cA_{X}$-module which is locally free
of rank $n$ as $\cO _X$-module and $v\in \Gamma (X,M)$ generates $M$
as an $\cA_{X}$-module. The equivalence relation is the obvious one.

A morphism of $k$-schemes $f:Y\to X$ induces an isomorphism of
sheaves of algebras
$$f^*\cA_{X}=\cO _Y\otimes _{f^{-1}(\cO _X)}f^{-1}\cA_{X}\to
\cA_{Y}$$ which gives the map of sets $f^*:{\bf M}^{[n]}_{\cA}(X)\to
{\bf M}^{[n]}_{\cA}(Y).$
\end{defi}

Note that we have the obvious forgetful morphism of functors
$\tilde{\bf {M}}^{[n]}_{\cA}\to {\bf M}^{[n]}_{\cA}$ which maps a
triple $(M,v,B)$ to the pair $(M,v).$

\begin{theo} \label{representability-hilbert} Fix $n\geq 1.$
The functor ${\bf M}^{[n]}_{\cA}$ is represented by a $k$-scheme
${\bf H}^{[n]}_{\cA}.$ If $k$ is noetherian then ${\bf
H}^{[n]}_{\cA}$ is of finite type over $k.$ The canonical morphism
$\tilde{{\bf H}}^{[n]}_{\cA}\to {\bf H}^{[n]}_{\cA}$ is a Zariski
locally trivial principal $\mathrm{GL}_n$-bundle
(\ref{def-of-locally-triv-G-bundle}).
\end{theo}

\begin{proof} We use the notation as in the proof of Theorem
\ref{representability-based-hilbert}. The $k$-group scheme
$\mathrm{GL}_n$ acts on the affine scheme $W\times V=\mathrm{Spec}
k[t^s_{ij},y_l]$ in a natural way. Namely if we arrange as above
the linear coordinates on the scheme $W\times V$ as $m$-tuples of
matrices and a column vector $(A_1,...,A_m,Y)$ then a matrix $g\in
\mathrm{GL}_n$ acts by the formula
$$g(A_1,...,A_m,Y)=
(gA_1g^{-1},...,gA_mg^{-1},gY)$$ The closed subscheme $\Rep
_{\cA}^{[n]}\times V$  is invariant under this action. Also each
affine open subscheme $U_{f_1...f_{n-1}} \subset W\times V$ is
$\mathrm{GL}_n$-invariant. Hence ${\bf U}\subset W\times V$ is
invariant, and therefore $\tilde{{\bf H}}^{[n]}_{\cA}$ is also
$\mathrm{GL}_n$-invariant.

Denote $U_{f_1...f_{n-1}}^{\cA}=U_{f_1...f_{n-1}}\cap \tilde{{\bf
H}}^{[n]}_{\cA}$ and consider the $\mathrm{GL}_n$-equivariant map
$\psi _{f_1...f_{n-1}}: U_{f_1...f_{n-1}}^{\cA}\to \mathrm{GL}_n$
which (in the above matrix notation) maps a point
$(a^1,...,a^m,b)$  to the matrix
$(b,f_1(a^1,...,a^m)b,...,f_{n-1}(a^1,...,a^m)b).$ It follows from
Proposition \ref{free-action-proposition} below that there exists
a categorical quotient
$$\pi :\tilde{{\bf H}}^{[n]}_{\cA}\to {\bf H}^{[n]}_{\cA}$$
which is a Zariski locally trivial principal
$\mathrm{GL}_n$-bundle. It remains to prove that the scheme ${\bf
H}^{[n]}_{\cA}$ represents the functor ${\bf M}^{[n]}_{\cA}$ and
that $\pi$ corresponds to the forgetful morphism of functors
$\tilde{\bf {M}}^{[n]}_{\cA}\to {\bf M}^{[n]}_{\cA}.$

First let us lift the $\mathrm{GL}_n$-action to the free $\cO
_{\tilde{\bf H}^{[n]}_{\cA}}$-module $M_0=k^n\otimes _k\cO
_{\tilde{\bf H}^{[n]}_{\cA}}.$ Recall that we represent an element
of $M_0$ as the dot product
$${\bf e}\bullet \alpha ^t=
(e_1,...,e_n)\bullet (\alpha _1,...,\alpha _n)^{t}$$ of the row
vector ${\bf e}=(e_1,...,e_n)$ and the column vector ${\bf
\alpha}^t=(\alpha _1,...,\alpha _n)^{t},$ where $\alpha _l\in
k[t^s_{ij},y_l].$ Then for $g\in \mathrm{GL}_n$ we define
$$g({\bf e}\bullet \alpha ^t):={\bf e}g^{-1}\bullet g(\alpha ^t),$$
where ${\bf e}\bullet g^{-1}$ is the matrix product and $g(\alpha
^t)=(g(\alpha _1),...,g(\alpha _n))^t.$ For $\beta \in
k[t^s_{ij},y_l]$ and $g\in \mathrm{GL}_n$ we have
$$g(\beta({\bf e}\bullet \alpha ^t))=g(\beta)g({\bf e}\bullet \alpha
^t)$$ so that $M_0$ is a $\mathrm{GL}_n$-equivariant vector bundle
on $\cO _{\tilde{\bf H}^{[n]}_{\cA}}.$

It follows from Proposition \ref{equivariant-sheaves} below that
there exists on ${\bf H}^{[n]}_{\cA}$ a unique (up to an
isomorphism) locally free sheaf $\cM$ with an isomorphism $\pi ^*
\cM\simeq M_0.$ This sheaf $\cM$ is isomorphic to $(\pi
_*M_0)^{\mathrm{GL}_n}$ - the subsheaf of $\pi _*M_0$ consisting
of $\mathrm{GL}_n$-invariant sections, i.e.
$\mathrm{GL}_n$-invariant sections of $M_0$ descend to sections of
$\cM.$ The next lemma provides many sections of the sheaf $\cM.$

\begin{lemma} \label{descent-of-universal-bundle}
(a) The element $v_0={\bf e}\bullet Y$ is
$\mathrm{GL}_n$-invariant.

(b) More generally, for any $f(x_1,...,x_m)\in \cA$ the element
$f(x_1,...,x_m)(v_0)$ is $\mathrm{GL}_n$-invariant.

(c) The invariant sections $f(x_1,...,x_m)(v_0)$ generate the sheaf
$\cM$ as an $\cO _{{\bf H}^{[n]}_{\cA}}$-module.

(d) The action of operators $x_i\in \cA$ on $M_0$ descends to an
action on $\cM,$ which makes it an $\cA _{{\bf
H}^{[n]}_{\cA}}$-module.

(e) Denote by $v_u\in \Gamma ({\bf H}^{[n]}_{\cA},\cM)$ the element
$v_0$ considered as a global section of $\cM.$  Then $v_u$ is a
generator of $\cM$ as an $\cA _{{\bf H}^{[n]}_{\cA}}$-module.
\end{lemma}

\begin{proof} (a) and (b). Recall that $f(x_1,...,x_m)({\bf
e}\bullet Y)={\bf e}\bullet f(A_1,...,A_m)Y,$ where
$f(A_1,...,A_m)$ is a square matrix with entries in
$k[t_{ij}^s,y_l]$ and $f(A_1,...,A_m)Y$ is the product of
matrices. Then by definition of the $\mathrm{GL}_n$-action on
$k[t_{ij}^s,y_l]$ and on $M_0$ we have
$$g(f(x_1,...,x_m)({\bf e}\bullet Y))={\bf e}g^{-1}\bullet
g(f(A_1,...,A_m)Y)={\bf e}g^{-1}\bullet gf(A_1,...,A_m)g^{-1} gY,$$
where $gf(A_1,...,A_m)g^{-1} gY$ is the matrix product of matrices
$g,f(A_1,...,A_m),g^{-1},g,Y.$ So
$$g(f(A_1,...,A_m)({\bf e}\bullet Y))={\bf e}\bullet
f(A_1,...,A_m)Y=f(x_1,...,x_m)({\bf e}\bullet Y).$$

c) Choose $f_1,...,f_{n-1}\in \cA$ and consider the corresponding
affine open subset $U_{f_1,...,f_{n-1}}^{\cA}\subset \tilde{\bf H}
^{[n]}_{\cA}$ as above. By Lemma \ref{free-action-lemma} there
exists a $\mathrm{GL}_n$-equivariant isomorphism
$U_{f_1,...,f_{n-1}}\simeq \mathrm{GL}_n\times \overline{U}$ for
an open subset $\overline{U}\subset {\bf H}^{[n]}_{\cA}.$ The
invariant sections
$${\bf e}\bullet Y,\ {\bf e}\bullet f_1(A_1,...,A_m)Y,\ ...,\
{\bf e}\bullet f_{n-1}(A_1,...,A_m)Y$$ form a basis of the
restriction of $M_0$ to $U_{f_1,...,f_{n-1}}.$ It follows that they
also form a basis for the restriction of $\cM$ to $\overline{U}.$

d) Recall that the $\cA$-module structure on $M_0$ is given by the
formula
$$x_s({\bf e}\bullet {\bf \alpha }^t)={\bf e}\bullet A_s{\bf \alpha }^t$$
Hence $x_i$ maps an invariant section ${\bf e}\bullet
f(A_1,...,A_m)Y$ to an invariant section  ${\bf e}\bullet
A_sf(A_1,...,A_m)Y.$ By c) this implies that $x_i$ preserves the
space of invariant sections. So the action of $x_i$ descends to
$\cM,$ which makes it an  $\cA _{{\bf H}^{[n]}_{\cA}}$-module.

e) This is now clear from the proof of c) and d).
\end{proof}

So $(\cM,v_u)$ is a pair on ${\bf H}^{[n]}_{\cA}$ as in Definition
\ref{defi-of-hilbert-functor} and clearly the pair $(M_0,v_0)$ on
$\tilde{\bf H}^{[n]}_{\cA}$ is the pullback of $(\cM,v_u)$ under the
map $\pi .$ It remains to show that ${\bf H}^{[n]}_{\cA}$ represents
the functor ${\bf M}^{[n]}_{\cA}.$  The pair $(\cM,v_u)$ defines a
morphism from the functor represented by ${\bf H}^{[n]}_{\cA}$ to
the functor ${\bf M}^{[n]}_{\cA}.$ Let us construct the inverse
morphism.

Let $Y$ be a scheme and $(N,w)$ be a pair on $Y$ representing an
element of ${\bf M}^{[n]}_{\cA}(Y).$ Choose an open covering
$\{V_i\}$ of $Y$ such that for each $i$ the restriction $N_i:=
N\vert _{V_i}$ is a free $\cO _{V_i}$ module of rank $n.$ Choose a
basis $B_i$ of $N_i$ and let $w_i\in N_i$ be the restriction of
$w.$ Then by Theorem \ref{representability-based-hilbert} the
triple $(N_i,w_i,B_i)$ defines a canonical map $\phi _i:V_i \to
\tilde{\bf H}^{[n]}_{\cA}.$ On the intersection $V_i\cap V_j$
there is a matrix $g\in \mathrm{GL}_n(\cO _{V_i \cap V_j})$ such
that $B_j=B_i\cdot g^{-1}$ (as usual we think of a basis as a row
vector). This matrix defines a map $g:V_i\cap V_j \to
\mathrm{GL}_n$ such that $g\cdot \phi _i=\phi _j.$ This shows that
the compositions $\pi \cdot \phi _i$ and $\pi \cdot \phi _j$ agree
on $V_i\cap V_j,$ hence we obtain the map $\psi :Y\to {\bf
H}^{[n]}_{\cA}$ such that the pullback of the pair $(\cM,v_u)$ is
isomorphic to $(N,w).$ This proves that the functor ${\bf
M}^{[n]}_{\cA}$ is represented by ${\bf H}^{[n]}_{\cA}.$

We also proved that the morphism
$\pi :\tilde{{\bf H}}^{[n]}_{\cA}\to {\bf H}^{[n]}_{\cA}$ is a
Zariski locally trivial principal $Gl _n$-bundle and it corresponds to the
canonical morphism of functors $\tilde{\bf {M}}^{[n]}_{\cA}\to {\bf
M}^{[n]}_{\cA}.$ This proves the theorem.
\end{proof}

\begin{example} \label{example-2} Let $n=1.$ We have seen in Example
\ref{example-1} above that $\tilde{{\bf H}}^{[1]}_{\cA}=Spec
\cA_{ab}\times \bbG _m.$ Then ${\bf H}^{[1]}_{\cA}=Spec\cA _{ab}$
and $\pi $ is the projection.
\end{example}

\begin{prop} \label{extension-scalars-H}
In the above notation let $k\to k^\prime $ be a
homomorphism of commutative algebras. Consider the algebra $\cA
^\prime=\cA \otimes _{k}k ^\prime.$ Then for any $n\geq 1$
$${\bf H}^{[n]}_{\cA ^\prime}={\bf H}^{[n]}_{\cA}
\times_kk^\prime.$$
\end{prop}

\begin{proof} By Proposition \ref{extension-scalars-based-hilbert}
we have $\tilde{\bf{H}}^{[n]}_{\cA
^\prime}=\tilde{\bf{H}}^{[n]}_{\cA} \times_kk^\prime.$ Also the
$\mathrm{GL}_n\times _kk^\prime$-action on
$\tilde{\bf{H}}^{[n]}_{\cA ^\prime}$ is obtained from the
$\mathrm{GL}_n$-action on $\tilde{\bf{H}}^{[n]}_{\cA }$ by
extending scalars from $k$ to $k^\prime .$ Since
$\tilde{\bf{H}}^{[n]}_{\cA }$ (resp. $\tilde{\bf{H}}^{[n]}_{\cA
^\prime}$) is a Zariski locally trivial principal
$\mathrm{GL}_n$-bundle over ${\bf{H}}^{[n]}_{\cA }$ (resp.
$\mathrm{GL}_n\times _kk^\prime $-bundle over ${\bf{H}}^{[n]}_{\cA
^\prime}$) we conclude that the scheme ${\bf{H}}^{[n]}_{\cA
^\prime}$ is obtained from ${\bf{H}}^{[n]}_{\cA }$ by extending
scalars from $k$ to $k^\prime .$
\end{proof}

\begin{remark} \label{closed-embedding-for-H}
Let $\phi :\cA \to \cB$ be a surjective homomorphism
of (finitely generated associative) $k$-algebras. This induces an obvious
morphism of functors $\phi _*:{\bf M}^{[n]}_{\cB}\to {\bf
M}^{[n]}_{\cA}$ by restriction of scalars from $\cB$ to $\cA.$ This
morphism of functors corresponds to a morphism of representing
schemes $\phi _*:{\bf H}^{[n]}_{\cB}\to {\bf H}^{[n]}_{\cA}.$ It
follows easily from the description of the scheme ${\bf  H}^{[n]}$ in
Theorem \ref{representability-hilbert} and from Remark
\ref{surjection=closed-embedding} that $\phi _*$ is a closed
embedding.
\end{remark}

\section{Projectivity of ${\bf H}^{[n]}_{\cA}$}

The universal bundle $\cM$ defines a cohomology class $\Delta \in
H^1({\bf H}^{[n]}_{\cA},\mathrm{GL}_n).$ We can explicitly
describe a 1-cocycle corresponding to $\Delta .$ Namely recall
that $\tilde{{\bf H}}^{[n]}_{\cA}$ has an affine open covering
consisting of $\mathrm{GL}_n$-invariant subsets
$U^{\cA}_{f_1...f_{n-1}}.$ Denote the collection
$f_1,...,f_{n-1}\subset \cA$ by ${\bf f}$ and put $U_{\bf
f}=U^{\cA}_{f_1...f_{n-1}},$ $ \overline{U}_{\bf f}=\pi (U_{\bf
f}),$ $M_{\bf f}=(Y,f_1(A_1,...,A_m)Y,...,f_{n-1}(A_1,...,A_m)Y),$
 $D_{\bf f}=D_{f_1...f_{n-1}}=\det M_{\bf f}.$ The affine open
 subsets $\overline{U}_{\bf f}$ cover ${\bf H}^{[n]}_{\cA}.$

Recall that for every $f(x_1,...,x_m)\in \cA$ the element ${\bf
e}\bullet f(A_1,...,A_m)Y\in M_0$ is $\mathrm{GL}_n$-invariant and
hence descends to a section of $\cM.$ Moreover $\cM$ is generated
by sections of this form. Let us define a trivialization
$$\cM \mid _{\overline{U}_{\bf f}}\stackrel{\sim}{\to}
\cO _{\overline{U}_{\bf f}}^{\oplus n},\quad {\bf e}\bullet
f(A_1,...,A_m)Y\mapsto M^{-1}_{\bf f}f(A_1,...,A_m)Y$$ where on
RHS we have the product of the matrix $M^{-1}_{\bf f}$ and the
column vector $f(A_1,...,A_m)Y.$ Note that this product consists
of $\mathrm{GL}_n$-invariant functions and hence is indeed an
element of $\cO _{\overline{U}_{\bf f}}^{\oplus n}.$ Therefore the
cohomology class $\Delta $ is represented by the 1-cocycle
$$\{ M^{-1}_{{\bf f}^\prime}\cdot M_{\bf f}:
\overline{U}_{\bf f}\cap \overline{U}_{{\bf f}^\prime}\to
\mathrm{GL}_n\}$$

Consider the line bundle $\cL=\bigwedge ^n\cM.$ It defines the
cohomology class $\delta \in H^1({\bf H}^{[n]}_{\cA},\bbG _m)$
which is the image of $\Delta $ under the map $H^1({\bf
H}^{[n]}_{\cA},\mathrm{GL}_n)\to H^1({\bf H}^{[n]}_{\cA},\bbG _m)$
induced by the homomorphism of $k$-group schemes $\det
:\mathrm{GL}_n \to \bbG _m.$ The class $\delta$ is described by
the cocycle
$$\{ \frac{D_{\bf f}}{D_{{\bf f}^\prime }}:
\overline{U}_{\bf f}\cap \overline{U}_{{\bf f}^\prime} \to \bbG
_m\}$$

For each $p\geq 0$ define the $k$-module
$$T_p=\{ h\in \Gamma (\tilde{\bf H}^{[n]}_{\cA},
\cO _{\tilde{\bf H}^{[n]}_{\cA}})\mid \forall g\in \mathrm{GL}_n,\
g(h)=(\det g)^p h\}$$ In particular $T_0=\Gamma (\tilde{\bf
H}^{[n]}_{\cA}, \cO _{\tilde{\bf H}^{[n]}_{\cA}})^{\mathrm{GL}_n}=
\Gamma ({\bf H}^{[n]}_{\cA}, \cO _{{\bf H}^{[n]}_{\cA}}).$ Also
notice that for each ${\bf f}\subset  \cA$ the function $D_{\bf
f}\in T_1.$
 Let $T=\oplus _{p\geq
0}T_p$ be the corresponding graded $k$-algebra.

\begin{lemma} \label{T=S}
There is a natural isomorphism of graded $k$-algebras
$$T_\bullet \simeq \bigoplus _{p\geq 0}H^0({\bf H}^{[n]}_{\cA},\cL^{\otimes
p}).$$ In particular the functions $D_{\bf f}$ correspond to
 sections of $\cL.$
\end{lemma}

\begin{proof} Let $h\in T_p.$ For each ${\bf f}$ the
function $h/D_{\bf f}^p$ on $U_{\bf f}$ is
$\mathrm{GL}_n$-invariant, so it descends to a function on
$\overline{U}_{\bf f}.$ Clearly the collection $\{h/D_{\bf f}^p\in
\cO (\overline{U}_{\bf f})\}$ defines a global section of
$\cL^{\otimes p}$ and this induces an injective homomorphism of
graded algebras
$$\theta : T_{\bullet}\to \bigoplus _{p\geq
0}H^0({\bf H}^{[n]}_{\cA},\cL^{\otimes p}).$$ Let us show the
surjectivity of $\theta .$ Let $\gamma \in \Gamma ({\bf
H}^{[n]}_{\cA},\cL^{\otimes p}).$ Then $\gamma $ corresponds to a
collection of sections $\{\gamma _{\bf f}\in \cO (\overline{U}_{\bf
f})\}$ such that $\gamma _{{\bf f}^\prime}= \frac{D_{\bf
f}^p}{D^p_{{\bf f}^\prime }}\gamma _{\bf f}$ on $\overline{U}_{\bf
f}\cap \overline{U}_{{\bf f}^\prime}.$ This gives us a global
function $h$ on $\tilde{{\bf H}}^{[n]}_{\cA}$ such that $h\mid
_{U_{\bf f}}=\pi ^* \gamma _{\bf f}D^p_{\bf f}.$ Clearly $h\in T_p$
and $\theta (h)=\gamma .$

\end{proof}

The line bundle has a description in terms of the universal bundle
$\cM$ on ${\bf H}^{[n]}_{\cA}.$

\begin{lemma} \label{compatibility-of-Mu}
Let $\cB$ be another associative $k$-algebra and $\cA \to \cB$ be a
surjection. Denote by $j:{\bf H}^{[n]}_{\cB}\hookrightarrow {\bf
H}^{[n]}_{\cA}$ be the induces closed embedding. Let $\cM^{\cA}$ and
$\cM^{\cB}$ (resp. $\cL^{\cA}$ and $\cL^{\cB}$) be the corresponding
universal bundles (resp. line bundles) on ${\bf H}^{[n]}_{\cA}$ and
${\bf H}^{[n]}_{\cB}$ respectively. Then $j^*M^{\cA}_u=M^{\cB}_u$
and $j^*\cL^{\cA}=\cL^{\cB}.$
\end{lemma}

\begin{proof} This follows from the commutativity of the natural diagram
$$\begin{array}{ccc}
\tilde{\bf H}^{[n]}_{\cB} & \hookrightarrow & \tilde{\bf H}^{[n]}_{\cA} \\
\downarrow & & \downarrow \\
{\bf H}^{[n]}_{\cB} & \stackrel{j}{\hookrightarrow} & {\bf H}^{[n]}_{\cA}
\end{array}
$$
and the definition of the universal bundle $\cM.$
\end{proof}

\begin{theo} \label{projectivity-for-free-algebra}
Assume that the ground ring $k$
is a UFD and a finitely generated algebra over a field, and let
$F_m=k\langle x_1,...,x_m\rangle$ be the free associative algebra.
Fix $n\geq 1$ and consider the graded $k$-algebra $S_\bullet
=S_\bullet ^{F_m}=\oplus _{p\geq
0}H^0({\bf H}^{[n]}_{F_m},\cL^{\otimes p}).$ Then

(i) The graded $k$-algebra $S_\bullet $ is finitely generated.
In particular the $k$-algebra $\cO({\bf H}^{[n]}_{F_m})=S_0$ is
finitely generated.

(ii) The line bundle $\cL$ on ${\bf H}^{[n]}_{F_m}$ is ample, i.e.
for some $n_0>0$ the sections $H^0({\bf H}^{[n]}_{F_m},\cL^{\otimes
n_0})$ define a closed embedding
$${\bf H}^{[n]}_{F_m}\hookrightarrow {\bf P}^N_{\cO({\bf H}^{[n]}_{F_m})}$$
for some $N\geq 0 .$
\end{theo}

\begin{proof} (i) We will denote ${\bf H}^{[n]}_{F_m}$
simply by ${\bf H}^{[n]}$, etc. In case $n=1,$ ${\bf H}^{[1]}=Spec
k[t^1,...,t^m]$ and the line bundle $\cL$ is trivial (Example
\ref{example-2}), hence the theorem holds with $N=0.$ So we may assume that
$n\geq 2.$

Also let us first consider the case $m=1,$ so that $F_1=k[x]$ - the
commutative polynomial ring over $k.$ Choose
$f_1=x,f_2=x^2,...,f_{n-1}=x^{n-1}\in F_1$ and consider
the corresponding affine open subset $U_{x...x^{n-1}}\subset
W\times V.$

\begin{lemma} We have the equality $U_{x...x^{n-1}}=\tilde{{\bf H}}^{[n]}.$
\end{lemma}

\begin{proof} Indeed, by definition $U_{x...x^{n-1}}\subset \tilde{{\bf H}}^{[n]}$
and both are open subschemes in $W\times V.$ Hence it suffices to show that they
have the same points. This is a consequence of the following simple
observation: let ${\bf F}$ be a field, $M$ - an $n$-dimensional ${\bf F}$-vector
space, $m\in M$ and $B\in \End _{\bf F}(M).$ We consider $M$ as a
${\bf F}[B]$-module in the obvious way. Then ${\bf F}[B]m=M$ if and only
if the ${\bf F}$-span of $m,Bm,...B^{n-1}m$ is equal to $M.$
\end{proof}

It follows that for $m=1,$ we have $\tilde{{\bf H}}^{[n]}=
Gl _n\times {\bf H}^{[n]}$ and ${\bf H}^{[n]}$ is an affine scheme.
Also the line bundle $\cL$ on ${\bf H}^{[n]}$ is trivial, hence
the theorem holds with $N=0.$ So we may assume that $m\geq 2.$

\begin{lemma} \label{equality-of-global-sections}
If $n,m\geq 2$ the restriction map
$\Gamma (W\times V,\cO _{W\times V})\to
\Gamma (\tilde{\bf H}^{[n]},\cO _{\tilde{\bf H}^{[n]}})$ is an isomorphism.
\end{lemma}

\begin{proof} The ring $\Gamma (W\times V,\cO _{W\times
V})=k [t^s_{ij},y_l]$ is a UFD. Recall that the
scheme $\tilde{\bf H}^{[n]}$ is the union of affine open subsets
$U_{f_1...f_{n-1}}$ defines by inverting the polynomials
$D_{f_1...f_{n-1}}.$ We can find two sets of elements
$\{f_1,...,f_{n-1}\},\{f^\prime _1,...,f^\prime _{n-1}\}\subset F_m$
such that the corresponding polynomials $D_{\bf f},$ $D_{{\bf f}^\prime}$
are not
equal to zero and are not proportional. It
follows that any element in the field of fractions of $\Gamma
(W\times V,\cO _{W\times V})$ which is a regular function on $U$ and
$U^\prime$ must be a polynomial. This proves the lemma.
\end{proof}

Now we can prove that the algebra $S_\bullet$ is finitely generated.
By Lemma \ref{T=S} there is an isomorphism of graded $k$-algebras
$S_\bullet \simeq T_\bullet,$ where
 $$T_p=\{ h\in \Gamma (\tilde{\bf H}^{[n]},
\cO _{\tilde{\bf H}^{[n]}})\mid \forall g\in \mathrm{GL}_n,\
g(h)=(\det g)^p h\}$$ By Lemma \ref{equality-of-global-sections}
 $$T_p=\{ h\in k[t^s_{ij},y_l]\mid \forall g\in \mathrm{GL}_n,\
g(h)=(\det g)^p h\}$$

Let $\mathrm{GL}_n$ act on $Spec k[t]$ by the formula $g(t)=\det
(g)^{-1} t.$ Consider the affine space $W\times V\times
\mathrm{Spec} k[t]$ with the diagonal $\mathrm{GL}_n$-action. Let
$B$ denote the (polynomial) algebra of global sections on $W\times
V\times \mathrm{Spec} k[t].$ Since the group scheme
$\mathrm{GL}_n$ is reductive and $k$ is of finite type over a
field we know by \cite{Se},Thm.2(i) that the $k$-algebra of
invariants $B^{\mathrm{GL}_n}$ is finitely generated. Note that
the algebra $B^{\mathrm{GL}_n}$ is graded
$$B^{\mathrm{GL}_n}=\bigoplus _{p\geq 0}T _p\cdot t^p.$$

Hence the graded algebra $B^{Gl _n}$ is isomorphic to $T_\bullet
\simeq S_\bullet .$ In particular $S_\bullet$ is also finitely
generated.

\begin{lemma} \label{generated-by-degree-one}
Let $k$ be a commutative ring  and $A_\bullet=\oplus _{p\geq 0}A_p$
be a finitely generated graded $k$-algebra. Then for some $m>0$ its
graded $A_0$-subalgebra $\oplus _{p\geq 0}A_{p m}$ is generated by
$A_m.$
\end{lemma}

\begin{proof} Let $x_1,\ldots,x_n$ be
generators of the $A_0$-algebra $A_\bullet $ of degrees
$d_1,\ldots,d_n$, and let $m = 2n D$, where $D = d_1\cdots d_n$.  We
claim that $A_m$ generates the $A_0$-algebra $\sum_p A_{pm}$.
Indeed, let $X = x_1^{a_1}\cdots x_n^{a_n}$ denote an element of
$A_{pm}$. Thus, $pm = \sum a_i d_i$.  We may assume $p\ge 2$, since
otherwise there is nothing to show. Let $y_j = x_j^{D/d_j}$, so all
$y_j$ have degree $D$. We claim that there exists an $n$-tuple of
non-negative integers $b_j$ such that $b_j D/d_j \le a_j$ for all
$j$ and $b_1+\cdots+b_n = 2n$.  If true, then $X$ can be written as
a product of $y_1^{b_1}\cdots y_n^{b_n}\in A_m$ and a monomial in
the $x_i$'s, and we win by induction on $p$.  To prove that the
$b_i$ exist, it suffices to show that $\sum \lfloor a_j d_j/D\rfloor
\ge 2n$. However,

$$\sum_j \lfloor a_j d_j/D\rfloor > \sum_j (a_j d_j/D-1) =
 pm/D - n \ge 2m/D - n = 3n > 2n.$$
\end{proof}

(ii) We now prove that the line bundle $\cL$ is ample.

Recall that for each ${\bf f}=\{f_1,...,f_{n-1}\}\subset  \cA$ the
function $D_{\bf f}$ on $\tilde{\bf H}^{[n]}$ belongs to $T_1\simeq
S_1,$ hence defines a global section of $\cL.$ We can choose a
finite set  $\{{\bf f}_i\},$ say
 such that the affine
open subsets $\overline{U}_{{\bf f}_i}=\{D_{{\bf f}_i}\neq 0\}$
cover ${\bf H}^{[n]}.$  Choose $n_0>0$ as in the above Lemma
\ref{generated-by-degree-one}. Suppose that the $S_0$-module
$H^0({\bf H}^{[n]},\cL^{\otimes n_0})$ is generated elements
$\{s_0,...,s_N\}$  and consider the corresponding morphism
$$\phi :{\bf H}^{[n]}\to {\bf P}^N_{S_0}$$ defined by
$H^0({\bf H}^{[n]},\cL^{\otimes n_0}).$ Notice that for each $i,$
$D^{n_0}_{{\bf f}_i}\in H^0({\bf H}^{[n]},\cL^{\otimes n_0}),$ hence
the morphism $\phi$ is well defined. We claim that $\phi $ is a
closed embedding.

\begin{lemma} Let $R$ be a commutative ring and $\phi :X\to
{\bf P}^{N}_R$ be a morphism of $R$-schemes corresponding to an
invertible sheaf $L$ on $X$ and sections $s_0,...,s_N\in \Gamma
(X,L).$ Suppose that there exists a subset $\{t_i\}\subset \{s_j\}$
such that

(i) each open subset $X_{t_i}=\{t_i\neq 0\}\subset X$ is affine;

(ii) for each $i,$ the map $R[y_0,...,y_N]\to \Gamma (X_{t_i},\cO
_{X_{t_i}})$ defined by $y_j\mapsto s_j/t_i$ is surjective.

Then $\phi $ is a closed embedding.
\end{lemma}

\begin{proof} This lemma is just a slight variation of
\cite{Ha},II,Proposition 7.2,
and the proof is the same.
\end{proof}

To apply this lemma  we may assume that the sections
$D^{n_0}_{{\bf f}_i}$'s are among the $s_j$'s. Pick one such
section $D^{n_0}_{\bf f}.$ As explained above the corresponding
open subset $\overline{U}_{\bf f}\subset {\bf H}^{[n]}$ is affine.
Choose $\rho \in \cO (\overline{U}_{{\bf f}}).$ Then $\pi
^{-1}(\rho )\in \cO (U_{{\bf f}}),$ and hence $\pi ^{-1}(\rho
)=h/D_{{\bf f}}^p$ for some $h\in \Gamma (\tilde{\bf H}^{[n]},\cO
_{\tilde{\bf H}^{[n]}})$ and $p\geq 0.$ We may assume that
$p=n_0p^\prime$ and so $h\in H^0({\bf H}^{[n]},\cL^{\otimes
n_0p\prime})$ (since $\pi^{-1}(\rho)$ is
$\mathrm{GL}_n$-invariant). But then $h$ is a polynomial in
$s_j$'s with coefficients in $S_0.$ This shows that the map
$S_0[s_0,...s_N]\to \cO (\overline{U}_{{\bf f}})$ defined by
$s_j\mapsto s_j/D^{n_0}_{\bf f}$ is surjective. Thus $\phi $ is a
closed embedding. This proves (ii) and the theorem.
\end{proof}

\begin{cor}\label{projectivity-for-general-algebra}
Let $k$ be as in Theorem
\ref{projectivity-for-free-algebra} and let
$\cA$ be a finitely generated associative $k$-algebra, say $\cA$ is a
quotient of $F_m.$ Fix $n\geq 1.$ Then for some $n_0>0$ the sections
$H^0({\bf H}^{[n]}_{\cA},\cL ^{\otimes n_0})$ define a closed embedding
$${\bf H}_{\cA}^{[n]}\hookrightarrow {\bf P}^N_{\cO({\bf H}^{[n]}_{F_m})}$$
for some $N\geq 0.$
\end{cor}

\begin{proof} Indeed, by Remark \ref{closed-embedding-for-H} the scheme
${\bf H}^{[n]}_{\cA}$ is a closed subscheme in ${\bf H}^{[n]}_{F_m}$ and
the line bundle $\cL ^{\cA}$ is the restriction of $\cL ^{F_m}$ (Lemma
\ref{compatibility-of-Mu}). Hence it remains to apply part (ii) of Theorem
\ref{projectivity-for-free-algebra}.
\end{proof}

\begin{remark} For a general $\cA$ as in Corollary
\ref{projectivity-for-general-algebra} we cannot prove that the $k$-algebra
$\cO ({\bf H}^{[n]}_{\cA})$ (or the $k$-algebra
$\cO (\tilde{{\bf H}}^{[n]}_{\cA})$) is finitely generated. This is the
reason for the appearance of the projective space
${\bf P}^N_{\cO({\bf H}^{[n]}_{F_m})}$ in that corollary.
\end{remark}

\subsection{The tangent sheaf of the Hilbert scheme} Assume that the base
ring $k$ is an algebraically closed field and let $\cA$ be a
finitely generated $k$-algebra. All schemes and morphism of schemes
are assumed to be defined over $k.$ Fix $n\geq 1.$ Put ${\bf H}={\bf
H}^{[n]}_{\cA}$ and consider the canonical short exact sequence of
$\cA_{\bf H}$-modules
\begin{equation}\label{fundam-short-exact-seq}
0\to \cI \to \cA_{\bf H}\to \cM \to 0
\end{equation}
where $\cM$ is the universal bundle and $\cI$ is the "universal left
ideal of codimension $n$ in $\cA$". In particular this is a short
exact sequence of quasi-coherent sheaves on ${\bf H}.$ Since the
$\cO _{\bf H}$-module $\cM$ is locally free this a locally split
sequence of quasi-coherent sheaves.

\begin{prop} \label{tangent-space-to-functor}
Let $Z$ be a scheme and $j:Z\to {\bf H}$ be
a morphism of schemes. Then there is a natural isomorphism of two
$\cO (Z)$-modules:

1) The set $T(Z)$ of morphisms $\nu :Z\times
Speck[\epsilon]/(\epsilon ^2) \to {\bf H}$ such that the composition
of the closed embedding $Z\to Z\times Speck[\epsilon]/(\epsilon ^2)$
with $\nu$ coincides with $j.$

2) The set of morphisms $\Hom _{\cA _Z}(j^*\cI ,j^*\cM ).$

In particular, for a closed point $x\in {\bf H}$ let $\cM _x$ be
the corresponding quotient of $\cA$ by the left ideal $\cI _x\subset
\cA.$ Then the $k$-vector space $\Hom _{\cA}(\cI _x,\cM _x)$ is
naturally isomorphic to the Zariski tangent space of ${\bf H}$ at
$x.$
\end{prop}

\begin{proof}
Put $\Lambda :=Speck[\epsilon]/(\epsilon ^2)$ and
let $i:Z\to Z\times \Lambda$ and $p:Z\times \Lambda \to Z$ be the
closed embedding and the projection. Choose $\phi \in T(Z).$ We
obtain two morphisms
$$jp,\phi :Z\times \Lambda \to {\bf H}$$
which induce the following diagram of $\cA _{Z\times
\Lambda}$-modules
$$\begin{array}{ccccc}
& & \phi ^*\cI & & \\
& & \downarrow & & \\
(jp)^*\cI  & \stackrel{\alpha }{\to } & \cA_{U\times \Lambda} &
\to & (jp)^*\cM \\
& & \downarrow \beta & & \\
& & \phi ^*\cM & &
\end{array}
$$
in which the horizontal (resp. the vertical) part is the pullback by
$jp$ (resp. $\phi $) of the universal sequence
\ref{fundam-short-exact-seq}. Since $\phi i=j$ the image of the
composition $\beta \alpha$ is contained in $\epsilon \phi ^*\cM
=j^*\cM .$ Hence this morphism factors through the quotient
$(jp)^*\cI/\epsilon (jp)^*\cI  =j^*\cI .$ So we obtain an element
$\beta \alpha \in \Hom _{\cA_Z}(j^*\cI ,j^*\cM ).$

Vice versa, let $\gamma \in \Hom _{\cA_Z}(j^*\cI ,j^*\cM ).$ Let
$U\subset Z$ be an affine open subset. We can choose a lift
$\tilde{\gamma}:(jp)^*\cI\mid _{U\times \Lambda} \to \cA_{U\times \Lambda}$
 of $p^*\gamma$ as a map
of $\cO _{U\times \Lambda}$-modules. Consider the ideal $I_U\subset \cA_{U\times
\Lambda}$ generated by elements $a-\epsilon \tilde{\gamma}(a)$ for
$a\in (jp)^*\cI \mid _{U\times \Lambda}.$
This ideal $I_U$ depends only on $\gamma$
and not on a choice of $\tilde{\gamma}.$ Hence we obtain a global
ideal $I\subset \cA _{Z\times \Lambda }.$  Put $M=\cA_{Z\times
\Lambda}/I.$ We have by construction $i^*I=j^*\cI $ as ideals in
$\cA_Z.$ Hence also $i^*M=j^*\cM $ as $\cA_ Z$-modules. Notice also
that $I\cap \epsilon \cA _{Z\times \Lambda}=\epsilon j^*\cI ,$ so
that the $\cO _{Z\times \Lambda}$-module $M$ is locally free. In
particular $M\in M^{[n]}_{\cA}(Z\times \Lambda),$ so that $M=\phi
^*\cM$ for a unique map $\phi :Z\times \Lambda \to {\bf H}.$ This
map $\phi $ is in $T(Z),$ which defines the inverse map $Hom _{\cA
_Z}(j^*\cI ,j^*\cM )\to T(Z).$ So we have established a natural
bijection of sets
$$T(Z)=\Hom _{\cA_Z}(j^*\cI ,j^*\cM )$$
One can check that this is a bijections of $\cO
(Z)$-modules. This proves the proposition.
\end{proof}

Let
$$T_{\bf H}:=\cH om _{\cO_{\bf H}}(\Omega ^1_{\bf H},\cO_{\bf H})$$
be the tangent sheaf of ${\bf H}.$
For any open subset $U\subset {\bf H}$ the sections $T_{\bf H}(U)$ are
the $k$-linear derivations of the sheaf $\cO _{U}.$
Since ${\bf H}$ is
of finite type over $k$ it follows that $T_{\bf H}$ is a
coherent sheaf \cite{EGA}, Corollaire 16.5.6.
The following is a variation on the Proposition
\ref{tangent-space-to-functor}.

\begin{prop} \label{tangent-sheaf}
There is a natural isomorphism of $\cO _{\bf
H}$-modules $T_{\bf H}=\cH om _{\cA_{\bf H}}(\cI ,\cM).$
\end{prop}

\begin{proof} Choose an affine open subset $U\subset {\bf H}.$ It
suffices to construct a natural isomorphism of $H^0(U,\cO
_U)$-modules
$$H^0(U,T_U)=\Hom _{\cA_U}(\cI _U,\cM _U)$$

\begin{lemma} Let $Y$ be a $k$-scheme. Put $\Lambda =Spec
k[\epsilon] /(\epsilon ^2)$ and consider the closed embedding
$i:Y\hookrightarrow Y\times \Lambda $ induces by the inclusion of
$Spec k$ in $\Lambda .$ Then there is a natural bijection between
the set of $k$-linear derivations $\delta :\cO _Y\to \cO _Y$
 and the set $S(Y)$ of morphisms of $k$-schemes $\mu: Y\times \Lambda
\to Y,$ such that $\mu i=\id _Y.$ This bijection respects the
$\cO (Y)$-module structure.
\end{lemma}

\begin{proof} Exercise.
\end{proof}

We apply the lemma to the case $Y=U.$
Since $U$ is open in ${\bf H}$ the set $S(U)=H^0(U,T_U)$ as in the lemma
coincides with
the set $T(U)$ of morphisms $\nu :U\times \Lambda \to {\bf H}$ such that
the composition $\nu i:U\to {\bf H}$ coincides with the embedding
$U\hookrightarrow {\bf H}.$ But by Proposition
\ref{tangent-space-to-functor} the $\cO(U)$-module $T(U)$ is
isomorphic to the $\cO(U)$-module $\Hom _{\cA_U}(\cI _U,\cM _U).$
This proves the proposition.
\end{proof}

\section{Some lemmas on free group actions}

In this subsection we prove some general simple statements about
schemes with a free group action and about equivariant sheaves on
them. Surely, this is "well known to experts" but we could not
locate a reference.

Fix a scheme $S.$ In this section all schemes are $S$-schemes. If a
group scheme $G$ acts on a scheme $X$ we denote by $m,p:G\times X\to
X$ the action and the projection morphisms respectively.

\begin{defi} \label{def-of-locally-triv-G-bundle}
Let $X$ be a scheme with an action of a group scheme
$G.$ A morphism $\pi :X\to \overline{X}$ is called a Zariski locally
trivial (resp. trivial) principal $G$-bundle if

1) $\pi \cdot m=\pi  \cdot p,$

2) there exists an open covering $\{W_i\}$ of $\overline{X}$ and for
each $i$ an isomorphism of $G$-schemes
$$G\times W_i\simeq \pi ^{-1}(W_i),$$
where the $G$-action on $G\times W_i$ is by left multiplication on
the first factor (resp. $G\times \overline{X}\simeq X$).
\end{defi}

\begin{lemma}\label{free-action-lemma}
Assume that a group scheme $G$ acts on a scheme $X.$ Suppose that
there exists a $G$-equivariant map $\sigma :X\to G$ (where $G$
acts on itself by left multiplication). Denote by $Z= \sigma
^{-1}(e)\subset X$ the closed subscheme, which is the preimage of
the identity $e\in G.$ Then there exists a $G$-equivariant
isomorphism $G\times Z\simeq X$ (so that the projection $X\to Z$
is a trivial principal $G$-bundle). If in addition the scheme $X$
is affine and $G$ is flat and quasi-compact, then $Z=\mathrm{Spec}
\Gamma (X,\cO _X)^G.$
\end{lemma}

\begin{proof} Let $i:Z\to X$ be the inclusion. We claim that the $G$-equivariant map
$$\theta :G\times Z\stackrel{(\id ,i)}{\longrightarrow}
G\times X\stackrel{m}{\longrightarrow} X$$ is an isomorphism.

Denote by $\sigma ^{-1}:X\to G$ the composition of $\sigma $ with
the inverse map $(-)^{-1}:G\to G.$

Consider the composition
$$h:X\stackrel{(\sigma ^{-1},\id)}{\longrightarrow} G\times X
\stackrel{m}{\longrightarrow} X,\quad x\mapsto \sigma ^{-1}(x)x.$$
Then $\sigma \cdot h$ maps $X$ to $e\in G.$ Therefore $h:X\to Z,$
hence we obtain the map $X\stackrel{(\sigma ,h)}{\longrightarrow}
G\times Z,$ which is the inverse of $\theta:$
$$x\stackrel{(\sigma ,h)}{\mapsto} (\sigma (x),\sigma
^{-1}(x)x)\stackrel{\theta}{\mapsto} \sigma (x)\sigma ^{-1}(x)x=x,$$
$$(g,z)\stackrel{\theta}{\mapsto} gz \stackrel{(\sigma ,h)}{\mapsto
}(\sigma (gz),\sigma ^{-1}(gz)gz)=(g\sigma (z),(g\sigma
(z))^{-1}gz)=(g,z).$$ So $\theta $ is an isomorphism.

Let us prove the last assertion. Since the group scheme $G$ is flat,
it is automatically faithfully flat. Hence the projection morphism
$G\times Z\to Z$ is faithfully flat and quasi-compact. Thus it is a
morphism of strict descent with respect to quasi-coherent sheaves.
It is explained in the proof of Proposition
\ref{equivariant-sheaves} below that this implies the equality
$$Z=\mathrm{Spec} \Gamma (X,\cO _X)^G.$$
\end{proof}

\begin{prop} \label{free-action-proposition}
Let a flat affine group scheme $G$ act on a separated scheme $X.$
Assume that $X$ has an affine open covering $\{ U_i\}$ with the
following property:

For each $i,$ $U_i$ is $G$-invariant and there exists a
$G$-equivariant map $\sigma _i:U_i\to G.$

Then there exists a universal categorical quotient $\pi :X\to Z$
which is a Zariski locally trivial principal $G$-bundle.
\end{prop}

\begin{proof}
Denote by $A_i$ the ring such that $U_i=\mathrm{Spec} A_i.$  It
follows from Lemma \ref{free-action-lemma} that for each $i$ there
exists a $G$-equivariant isomorphism $U_i\simeq G\times Z_i,$
where $Z_i=\sigma ^{-1}_i(e)$ is the (closed subscheme of $U_i$)
preimage of the identity $e\in G.$ Moreover, there is a natural
isomorphism $Z_i=\mathrm{Spec} (A_i^G).$ For two indices $i,j$ we
can apply this lemma to the $G$-invariant affine open subscheme
$U_i\cap U_j$ and the restriction of each of the $G$-equivariant
maps $\sigma _i,\sigma _j.$ We conclude that the affine schemes
$Z_i\cap U_j$ and $Z_j\cap U_i$ are canonically isomorphic to
$\mathrm{Spec} \Gamma(U_i\cap U_j,\cO _{U_i\cap U_j})^G.$ Hence
there is a canonical isomorphism $\phi _{ij}:Z_i\cap
U_j\stackrel{\sim}{\longrightarrow}Z_j\cap U_i.$ Clearly, the
isomorphisms $\{\phi _{ij}\}$ satisfy the cocycle condition and
hence there is a scheme $Z$ which is obtained by gluing the affine
schemes $Z_i$ along the open subschemes $Z_i\cap U_j$ using the
isomorphisms $\phi _{ij}.$

We have the canonical morphism of schemes $\pi :X\to Z$ which
locally on each $U_i$ is induced by the inclusion of rings $A^G_i
\hookrightarrow A_i.$ Clearly $\pi \cdot m=\pi \cdot p.$  If
we identify $U_i=G\times Z_i,$ then the restriction of the map $\pi$
to $U_i$ is the projection on the second factor. Hence $\pi$ is a
Zariski locally trivial principal $G$-bundle.

Obviously the map $\pi \vert _{U_i}:U_i\to Z_i$ is a universal
categorical quotient. Hence the map $\pi$ is also such. This proves
the proposition.
\end{proof}

Next we consider equivariant sheaves on schemes with free group
actions.

For a scheme $Y$ we denote by $Mod(Y)$ the category of $\cO
_Y$-modules and by $\mathrm{Qcoh}(Y)\subset Mod(Y)$ its full
subcategory of quasi-coherent sheaves on $Y.$

Let $G$ be a group scheme acting on a scheme $X.$ Consider the
diagram of schemes

\begin{equation}\label{simplicial-equivariant-diagram}
  \xymatrix{
    {G\times G\times X}
    \ar@<2.5ex>[r]^-{d_0}
    \ar@<-2.5ex>[r]^-{d_2}
    \ar[r]^-{d_1}
    &
   {G\times X}
    \ar@<2.5ex>[r]^-{d_0}
    \ar@<-2.5ex>[r]^-{d_1}
    &
    {X}
    \ar[l]_-{s_0}
    }
\end{equation}
 where
$$d_0(g_1,...,g_n,x)=(g_2,...,g_n,g_1^{-1}x),$$
$$d_i(g_1,...,g_n,x)=(g_1,...,g_ig_{i+1},...,g_n,x),\quad 1\leq
i\leq n-1,$$
$$d_n(g_1,...,g_n,x)=(g_1,...,g_{n-1},x),$$
$$s_0(x)=(e,x).$$
(Here $e:S\to G$ is the identity of $G.$)

Recall that a $G$-equivariant quasi-coherent sheaf on $X$ is a
pair $(F,\theta ),$ where $F\in \mathrm{Qcoh}(X)$ and $\theta $ is
an isomorphism
$$\theta :d_1^*F\stackrel{\sim}{\to} d_0^*F$$
satisfying the cocycle condition
$$d_0^*\theta \cdot d_2^*\theta =d^*_1\theta,\quad s^*_0\theta =\id
_ F.$$ A morphism of equivariant sheaves is a morphism of
quasi-coherent sheaves $F\to F^\prime$ which commutes with
$\theta.$ Denote by $\\mathrm{Qcoh} _G(X)$ the category of
$G$-equivariant quasi-coherent sheaves on $X.$

\begin{remark} Let $Y$ be another scheme and
let $\sigma :X\to Y$ be a morphism such that $\sigma \cdot
d_0=\sigma \cdot d_1$ (equivalently $\sigma \cdot m=\sigma \cdot
p.$) Then we obtain the functor $\sigma ^*:\mathrm{Qcoh}(Y)\to
\mathrm{Qcoh}_G(X),$  $A\mapsto (\sigma ^*A,\id ).$ There exists a
natural functor $\sigma _*^G:\mathrm{Qcoh} _G(X)\to Mod (Y),$
which is the right adjoint to $\sigma ^*.$ Namely for
$(F,\theta)\in \mathrm{Qcoh} _G(X)$ and an open subset $U\subset
Y$ put
$$\sigma _*^G(F,\theta )(U)=\Gamma (\sigma ^{-1}(U),F)^G
:=\{s\in F(\sigma ^{-1}(U))\vert d_0^*s=\theta (d^*_1s)\}.$$
\end{remark}

\begin{prop} \label{equivariant-sheaves} Let $G$ be a quasi-compact
and flat group scheme acting on a scheme $X.$ Let $\pi :X\to
\overline{X}$ be a Zariski locally trivial principal  $G$-bundle
(Definition \ref{def-of-locally-triv-G-bundle}). Then the functor
$\pi ^*:\mathrm{Qcoh}(\overline{X})\to \mathrm{Qcoh} _G(X)$ is an
equivalence of categories. This equivalence preserves locally free
sheaves.
\end{prop}

\begin{proof} Case 1. Assume that $\pi$ is a trivial principal
homogeneous $G$-bundle, i.e. $X=G\times \overline{X}.$ In this
case the proposition follows from the faithfully flat descent for
quasi-coherent sheaves. Indeed, it is easy to see that in this
case the diagram \ref{simplicial-equivariant-diagram} above is
isomorphic to one where all the arrows $d_i$ become projections
and then a $G$-equivariant sheaf $(F,\theta)$ is the same as a
descent data for $F\in \mathrm{Qcoh}(X)$ relative to the
faithfully flat and quasi-compact morphism $\pi :X=G\times
\overline{X}\to \overline {X}.$ So it remains to apply the
corresponding theorem of Grothendieck \cite{SGA},Expose
VIII,Thm.1.1.

Case 2. This is the general case. We first show that the functor
$\sigma ^*:\mathrm{Qcoh}(\overline{X})\to \mathrm{Qcoh} _G(X)$ is
full and faithful. Let $A\in \mathrm{Qcoh} (\overline{X}).$ It
suffices to show that the adjunction morphism $\alpha :A\to \sigma
_*^G\sigma ^*A$ is an isomorphism. But this question is local on
$\overline{X},$ hence we may assume that $\pi$ is a trivial
principal $G$-bundle. Thus $\alpha$ is an isomorphism by Case 1.

Now we can prove that $\sigma ^*$ is essentially surjective. Let
$(F,\theta)\in \mathrm{Qcoh}_G(X)$ and choose an open covering
$\{W_i\}$ of $\overline{X}$ such that
$$\pi \vert _{\pi ^{-1}(W_i)}:\pi ^{-1}(W_i)\to W_i$$
is a trivial principal $G$-bundle. Let $(F_i,\theta _i)\in
\mathrm{Qcoh} _G(\pi ^{-1}(W_i))$ be the restriction of $(F,\theta
)$ to $\pi ^{-1}(W_i).$ By Case 1 there exists a object $A_i\in
\mathrm{Qcoh} (W_i)$ and an isomorphism $\beta _i:\pi
^*(A_i)\stackrel{\sim}{\to} (F_i,\theta _i).$ Also for each pair
of indices $(i,j)$ there exists a unique isomorphism
$$\phi _{ij}:A_i\vert _{W_i\cap W_j}\stackrel{\sim}{\to} A_j\vert
_{W_i \cap W_j}$$ such that the following diagram commutes
$$\begin{array}{rcl}
\pi ^*(A_i\vert _{W_i\cap W_j}) & \stackrel{\pi ^*\phi _{ij}}
{\longrightarrow} & \pi ^*(A_j\vert _{W_i\cap W_j})\\
\beta _i\vert _{W_i\cap W_j} \downarrow & & \downarrow \beta _j\vert
_{W_i\cap W_j}\\
(F_i,\theta _i)\vert _{W_i\cap W_j} & = & (F_j,\theta _j)\vert
_{W_i\cap W_j}
\end{array}
$$
The uniqueness of $\phi _{ij}$ means that the collection $\{\phi
_{ij}\}$ satisfies the cocycle condition, so that $A_i^s$ come
from a global $A\in \mathrm{Qcoh} (\overline{X}).$ This proves the
proposition.
\end{proof}

\end{document}